\documentclass[10pt]{amsart}

\usepackage{amsthm, amsfonts, amssymb, color}
 \usepackage{mathrsfs}
\usepackage{amsmath}
 \usepackage{amstext, amsxtra}
  \usepackage{txfonts}
 \usepackage[colorlinks, linkcolor=black, citecolor=blue, pagebackref, hypertexnames=false]{hyperref}
\usepackage{graphicx}
\usepackage{overpic}

 \allowdisplaybreaks
\newtheorem{cor}{\hskip\parindent {Corollary}}[section]

\newtheorem{theorem}{Theorem}[section]

\newtheorem{lemma}[theorem]{Lemma}
\newtheorem{definition}[theorem]{Definition}
\newtheorem{example}[theorem]{Example}
\newtheorem{remark}[theorem]{Remark}

\title[Navier-Stokes equaiton and Symmetry]{
Symmetry and rigidity: Only one kind of symmetry allow non-zero real symmetric solution}
\author{Qixiang Yang}

\address {Qixiang Yang, School of Mathematics and Statistics, Wuhan University, Wuhan, 430072, P. R. China}
\email{qxyang@whu.edu.cn}
\date{\today}

\begin{document}

\maketitle
\begin{abstract}

In this paper, I study the symmetric structure on space variable $x$ for linear terms and non linear terms
appeared in the Navier-Stokes equations.
There exist {\bf 262144 kinds} of symmetric properties for complex vector fields.
There exist many factors which can change the symmetry properties.
I have found {\bf some methods to properly classify symmetry and combine symmetry-related terms}.
So I can study systematically all the symmetries and
prove that {\bf only one } kind of symmetry allow non-zero real symmetric solution.
By the way, we apply such symmetric result to the Navier-Stokes equations on the domain and
I prove the existence of smooth solution with energy conservation.

\end{abstract}


\section{Motivations}
\setcounter{equation}{0}
The Cauchy problem of the incompressible Navier-Stokes equations on the half-space $\mathbb{R}_{+}\times \mathbb{R}^{3}$ is given as:
\begin{equation}\label{1.1}
\left\{ \begin{aligned}
  &\partial_t u -\Delta u+u\cdot\nabla u-\nabla p=0,\;\;(t,x) \in \mathbb{R}_{+}\times \mathbb{R}^{3}, \\
  &\nabla\cdot u=0,\\
  u&(0, x)=u_{0}(x).
 \end{aligned}\right.
 \end{equation}
In this paper, I consider the symmetric structure on space variable $x$ for linear terms and non linear terms.
Further, I consider the relation between  rigidity and symmetry.
By the way, I consider an application to the relative equations on the domain.

Denote
\begin{equation}\label{1.2}
\left\{
\begin{array}{ccl}
A(u,v)&\equiv &     u\cdot\nabla v 
;\\
G(u,v)&\equiv &\nabla A(u,v)= \sum_{l, l'} \partial_{l} \partial_{l'}(u_{l}v_{l'});\\
C(u,v) &\equiv & \mathbb{P}\nabla(u\otimes v)= A(u,v)+(-\Delta)^{-1} \nabla G(u,v);\\
B(u, v)(t, x)&\equiv &\displaystyle\int_{0}^{t}e^{ (t-s) \Delta }\mathbb{P}\nabla(u\otimes v)ds,
\end{array}
\right.
\end{equation}
A solution of the above Cauchy problem (\ref{1.1}) is then obtained via the integral equation
\begin{equation}\label{1.3}
\begin{split}
u(t, x)&=e^{t\Delta}u_{0}(x)-B(u, u)(t, x).
\end{split}
\end{equation}
Which can be solved by a fixed-point method whenever the convergence is suitably defined in some function space.
Denote
\begin{align}\label{1.4}
\begin{array}{cccl}
u^{0}(t,x)&=& e^{t\Delta} u_0\,\, ;& \\
u^{\tau+1}(t,x)&=&  u^{0}(t,x) - B (u^{\tau}, u^{\tau})(t,x),&\forall \tau=0,1,2,\cdots.
\end{array}
\end{align}
For $u_0\in X_0^{3}$, if there exists $X^{3}$ such that $e^{t\Delta} u_0 \in X^{3}$
and $u^{\tau}$ converge to some function $u(t,x)\in X^{3}$, then $u(t,x)$ is the solution of \eqref{1.3} and $u(t,x)$
is called to be the mild solutions of \eqref{1.1}.
The notion of mild solution was pioneered by Kato-Fujita \cite{KatF} in 1960s.

During the latest decades, many important results about the mild solutions of \eqref{1.1} have been established;
see for example, Cannone \cite{Canj, Canb}, Germain-Pavlovic-Staffilani \cite{Ger}, Giga-Inui-Mahalov-Saal\cite{GIMS},
Giga-Miyakawa \cite{Gig}, Kato \cite{Kat}, Koch-Tataru \cite{Koc}, Lei-Lin \cite{Lei}, Wu \cite{Wu1, Wu2, Wu3, Wu4}, Xiao \cite{Xia1, Xia2}
and their references including Kato-Ponce \cite{KatP} and Taylor \cite{Tay} (see also the book \cite{Lem} ).
Further, applying the wavelets,
mild solutions have extended to many function spaces: Trieble-Lizorkin spaces,
Besov Morrey spaces and Trieble-Lizorkin Morrey spaces (see \cite{LiX, LiY1, LiY2, Lin, YanL, YanY} ).
Among all these results,
Koch-Tataru's ${\rm BMO}^{-1}$ is the biggest initial data spaces for non-phase spaces cases.
Giga-Inui-Mahalov-Saal and Lei-Lin studied   $P^{-1}_{1,1}$, which is the biggest initial data spaces for phase spaces cases,
see also \cite{Yang}.
There are also several results on illposedness.
See Bourgain-Pavlovi\'c \cite{BP} and Yoneda \cite{Yo}

Yang-Yang-Wu \cite{YYW} found that {\bf symmetry plays a key role in the control of the superimposed effect}
when studying the illposedness of the above equations \eqref{1.1}. See Remark \ref{re.ill} in the below section \S 2.2.
Yang \cite{Yang} found that {\bf the iterative algorithm \eqref{1.4} always converges at some points for symmetric solution}.
In \cite{Yang}, I have used Fourier transformation to study only a few part of the symmetry property on spatial variables $x$.
I found that certain symmetric properties did not produce symmetric solutions,
and found also that certain produced symmetric solutions.
There's a lot of symmetry and a lot of factors that affect symmetry.
In this paper, I have found {\bf some methods to properly classify symmetry and combine symmetry-related terms}.
So I can study {\bf systematically all the symmetries and got rigidity}.
Concretely speaking, I consider the following four terms:
\begin{itemize}
\item[(1)] There exists {\bf 262144 different kinds} of symmetric properties for complex vector fields.
In section 3, I find out all the symmetric solenoidal vector fields.
There exist {\bf 30 different kinds} for real vector fields.
There exist {\bf 984 different kinds} for complex fields.

\item[(2)] For complex vector fields $u$ and $v$, $B(u,v)$ can be written as {\bf the sum of 96 terms}.
Each  of these terms involves {\bf products, derivatives, integrals, operator actions
$e^{(t-s)\Delta}$} associated with two real symmetric functions, some terms also involve operator actions of $(-\Delta)^{-1}$.
{\bf Many factors may change the symmetry}.
For two symmetric vector fields $u$ and $v$ with divergence zero, $B(u,v)$ is {\bf not always symmetric vector fields}.
But for two arbitrary solenoidal vector fields $u$ and $v$, in section 4,
I can prove that there exist {\bf matched } symmetric vector fields $(u^{\alpha}, v^{\beta})_{\alpha,\beta\in \{0,1\}^{3}}$
such that $u=\sum\limits_{\alpha\in \{0,1\}^{3}} u^{\alpha}, v=\sum\limits_{\beta\in \{0,1\}^{3}} v^{\beta}$ and
all $B(u^{\alpha}, v^{\beta})$ are symmetric solenoidal vector fields.

\item[(3)] The rigidity has been studied extensively for partial differential equations.
For incompressible Navier-Stokes equations,
Leray speculated that, a blow-up solution should have similar structure as its initial data and proposed to consider self-similar solution.
Necas-Ruzicka-Sverak \cite{NRS} proved in 1996 that the only possible self-similar solution is zero.
That is to say, there exists rigidity phenomenon for self similar solution.
In sections 5 and 6, I consider mainly the relation between symmetry and rigidity.

{\bf The existence of Beltrami flow makes the symmetric situation different from the self-similar cases.}
Constantin-Majda \cite{CM} and Lei-Lin-Zhou \cite{LLZ} have constructed some vector fields $u_0$ such that
\begin{equation}\label{1.n}
\left\{ \begin{aligned}
  &\mathbb{P} \nabla (e^{t\Delta} u_0, e^{t\Delta} u_0)=0 & \mbox{ in } (\mathcal{S}'(\mathbb{R}^{3}))^{3}, \forall t>0, \\
  &\nabla\cdot u_0=0 & \mbox{ in } (\mathcal{S}'(\mathbb{R}^{3}))^{3}.
\end{aligned}\right.
\end{equation}
So we consider non Beltrami flow.
In section 5, I prove that, if $u_0$ is real vector field which does not satisfy \eqref{1.n},
then there exists {\bf only one} symmetric property allow non-zero symmetric solution for equations \eqref{1.1}.
In section 6, I prove that, if $u_0$ is complex vector field which does not satisfy \eqref{1.n},
then there exists {\bf only eight} symmetric property allow non-zero symmetric solution for equations \eqref{1.1}.

\item[(4)] Lastly, in section7, I apply such symmetric result to consider the Navier-Stokes equations on the domain and
I prove the existence of {\bf smooth solution with energy conservation}.

\end{itemize}

At the end of this section, I introduce some notations which will be used through out this paper.
$\forall \alpha\in \mathbb{N}$, denote $m(\alpha) = \alpha (\rm mod 2)\in \{0,1\}.$ Further,
$\forall \alpha=(\alpha_1,\alpha_2,\alpha_3)\in \mathbb{N}^{3}$, denote
$$m(\alpha) = (\alpha_1 (\rm mod 2),\alpha_2 (\rm mod 2),\alpha_3 (\rm mod 2))\in \{0,1\}^{3}.$$

\section{Preliminaries}
\setcounter{equation}{0}

In this section, I recall some  necessary preliminaries on mild solution
and present some preliminaries on symmetry.
There are two category of symmetry for velocity field.
One is the symmetry on the component of velocity field which has been studied by many people.
See Abidi-Zhang \cite{AbiZ} and Yang \cite{Yang}.
Another is the symmetry for the independent variables of velocity field.
Yang-Yang-Wu \cite{YYW} found that the latter symmetric structure causes some
superimposed effect in the study of ill-posedness of \eqref{1.1}.
Hence Yang \cite{Yang} has applied Fourier transform to consider a party of symmetric properties.

\subsection{Mild solution}

Mild solution of \eqref{1.1} is based on \eqref{1.4} and has been studied by many people in different initial data spaces.
See \cite{AbiZ, Canj, Canb, Ger, GIMS, Gig, Kat, KatF, KatP, Koc, Lei, LLZ, Lem, LiX, LiY1, LiY2, Lin, Tay, Wu1, Wu2, Wu3, Wu4, Xia1, Xia2,
Yang, YanL, YanY}.
I recall a smooth solution result for Sobolev space studied in \cite{LiX} by wavelets and
which will be used in the study of wellposedness on the domain.
In this paper, we use Meyer wavelets, see \cite{Mey}.
Let $\Phi^{0}(x)$ be the father wavelet and $\forall \epsilon\in \{0,1\}^{3}\backslash 0$,
let $\Phi^{\epsilon}(x)$ be the mother wavelet. Denote
$$\Lambda = \{(\epsilon,j,k), \epsilon\in \{0,1\}^{3}\backslash 0, j\in \mathbb{Z}, k\in \mathbb{Z}^{3}\}.$$
Further,
$\forall (\epsilon,j,k)\in \Lambda$, denote $\Phi^{\epsilon}_{j,k}(x) = 2^{\frac{3j}{2}} \Phi^{\epsilon}(2^{j}x-k).$
For $(\epsilon, j,k) \in \Lambda$, let $a^{\epsilon}_{j,k} = \langle f, \Phi^{\epsilon}_{j,k}\rangle$. Then
$$f(x)= \sum_{(\epsilon, j,k)\in\Lambda} a^{\epsilon}_{j,k} \Phi^{\epsilon}_{j,k}(x).$$
The property of Sobolev space can be characterized in the following way, see \cite{LiX,Lin,Mey,YanY}:
\begin{lemma}
$f(x)\in \dot{H}^{\frac{1}{2}}(\mathbb{R}^{3})$ if and only if
$$\sum\limits_{(\epsilon,j,k)\in \Lambda }
2^{j} |a^{\epsilon}_{j,k}|^{2}
<\infty.$$
\end{lemma}

For $t>0$, denote $j_t$ the smallest integer such that $2^{2j_t}\geq 1$.
For $(\epsilon, j,k) \in \Lambda$, let $a^{\epsilon}_{j,k}(t) =
\langle f(t,\cdot), \Phi^{\epsilon}_{j,k}\rangle$. Then
$$f(t,x)= \sum_{(\epsilon, j,k)\in\Lambda} a^{\epsilon}_{j,k}(t) \Phi^{\epsilon}_{j,k}(x).$$
The following solution space definition can reflect
the different role of high frequency part and low frequency part,
and which can be found in \cite{LiX}.
\begin{definition} $ m,m'>0$
$f(t,x)\in S_{m,m'}$, if the high frequency part of $f(t,x)$ satisfies
\begin{equation}\label{eq:ss1}
w^{h}_{m,\infty}(f)= \sup\limits_{t>0}
t^{\frac{m}{2}}\{\sum\limits_{(\epsilon,j,k)\in \Lambda, j\geq j_t }
2^{2 (m+\frac{1}{2})j}
|a^{\epsilon}_{j,k}(t)|^{2}\}^{\frac{1}{2}}
<\infty.
\end{equation}
\begin{equation*}
w^{h}_{m',2}(f)= \{\int t^{m-1}
\sum\limits_{(\epsilon,j,k)\in \Lambda, j\geq j_t}
2^{2 (m+\frac{1}{2})j}
|a^{\epsilon}_{j,k}(t)|^{2} dt \}^{\frac{1}{2}}
<\infty.
\end{equation*}
and the low frequency part of $f(t,x)$ satisfies
\begin{equation*}
w^{l}_{\infty}(f)= \{\sup\limits_{t>0}
\sum\limits_{(\epsilon,j,k)\in \Lambda, j< j_t }
2^{j} |a^{\epsilon}_{j,k}(t)|^{2}\}^{\frac{1}{2}}
<\infty.
\end{equation*}
\begin{equation*}
w^{l}_{m',2}(f)= \{\int t^{m'-1}
\sum\limits_{(\epsilon,j,k)\in \Lambda, j<j_t}
2^{2 (m'+\frac{1}{2})j}
|a^{\epsilon}_{j,k}(t)|^{2} dt
\}^{\frac{1}{2}}<\infty.
\end{equation*}
\end{definition}

For $m>1$, the equation \eqref{eq:ss1} means the high frequency part of $f(t,x)$ belongs to $C^{m-1}(\mathbb{R}^{3})$.
Hence $S_{m,m'}(\mathbb{R}^{3})\subset C^{m-1}(\mathbb{R}^{3})$.
The following Lemma \ref{th:LXY} is a restate of the particular case of theorem 1.1 in Li-Xiao-Yang \cite{LiX}.
\begin{lemma}\label{th:LXY}
Given $m>2, 0<m'<1$ and given
$u_0\in (\dot{H}^{\frac{1}{2}})^{3}$ satisfying ${\rm div}\,  u_0=0$.
Then $u^{\tau}(t,x)$ defined in \eqref{1.4} belongs to $S_{m,m'}^{3}$.
If\, $\|u_0\|_{ (\dot{H}^{\frac{1}{2}})^{3}}$ is small,
then $u^{\tau}(t,x)$ converge to the solution $u(t,x)$
of Navier-Stokes equations (\ref{1.1}) and $u(t,x) \in S_{m,m'}^{3}$
satisfying  ${\rm div}\, u(t,x)=0$.
\end{lemma}

\subsection{Symmetry of complex functions}
There are two category of symmetry for velocity field.
One is the symmetry on the component of velocity field which has been studied by many people.
See Abidi-Zhang \cite{AbiZ} and Yang \cite{Yang}.
Another is the symmetry for the independent variables of velocity field.
When Yang-Yang-Wu study the illposedness of \eqref{1.1} in \cite{YYW},
they found symmetry of independent variables can produce superimposed effect.
\begin{remark}\label{re.ill}
After completing \cite{YYW}, we understand that what actually leads to the illposedness of Navier-Stokes equations
is the following three properties: symmetry, monotony and some similarity of different frequencies.
These properties lead to a regular superposition for nonlinear term.
Since we have finished the redaction of \cite{YYW},
we did not modify the notations and kept the original appearance.
In fact, if we shift the initial value $u_0$ at point $\frac{3e}{2}$,
change $g(t,x)$ appropriately to be another antisymmetric function,
this is why I can find the relative symmetry structure in \cite{Yang}.
In the present paper, we take a closer look at all symmetrical structures.
\end{remark}
In this section, I consider some preliminaries on symmetry of independent variables for complex functions
$u(x)=u^{re}(x)+iu^{im}(x)$ where the real functions $u^{re}(x)$ and $u^{im}(x)$
all have symmetry or anti-symmetry properties with respect to all their independent variables.
A function has symmetry property means
\begin{definition}
(i) A real function $u(x) $ has symmetry property, if there exists function $f:\{0,1\}^3
\rightarrow \{0,1\}$ such that $(u,f) $  satisfies,
\begin{equation}\label{eq:sr1}
\begin{array}{c}
u( (-1)^{\alpha_1} x_1, (-1)^{\alpha_2} x_2, (-1)^{\alpha_3} x_3)
= (-1)^{f( \alpha)} u^{re} (x_1, x_2, x_3),\\
\forall \alpha=(\alpha_1,\alpha_2, \alpha_3)\in \{0,1\}^3,
x= (x_1, x_2, x_3)\in \mathbb{R}^{3}.
\end{array}
\end{equation}

(ii) A complex function $u(x)= u^{re}+i u^{im} $ has symmetry property,
if its real part and imaginary part both satisfy symmetry property.
Hence, there exists two functions $f,g:\{0,1\}^3
\rightarrow \{0,1\}$ such that
$(u^{re},f)$ and  $(u^{im},g)$  satisfy \eqref{eq:sr1}.
That is to say,
$\forall \alpha= (\alpha_1,\alpha_2, \alpha_3), \beta= (\beta_{1}, \beta_{2}, \beta_{3})\in \{0,1\}^3$,
$\forall (x_1, x_2, x_3)\in \mathbb{R}^{3}$,
\begin{equation}\label{eq:sr2}
u^{re}( (-1)^{\alpha_1} x_1, (-1)^{\alpha_2} x_2, (-1)^{\alpha_3} x_3)
= (-1)^{f( \alpha)} u^{re} (x_1, x_2, x_3).
\end{equation}
\begin{equation}\label{eq:si3}u^{im}( (-1)^{\beta_1} x_1, (-1)^{\beta_2} x_2, (-1)^{\beta_3} x_3)
= (-1)^{g( \beta)} u^{im} (x_1, x_2, x_3).
\end{equation}
\end{definition}

To simplify the notations, I introduce operator $T$ for symmetry functions:
\begin{definition} Given $\alpha, \beta\in \{0,1\}^{3}$.
If $u(x)$ have the same symmetry property as functions $x^{\alpha}+ix^{\beta}$, denote $$Tu(x)= \alpha+i \beta.$$
If $u(x)$ is a real function, denote also $Tu(x)= \alpha$. Further, if $\beta=0$, the notation
$i\,0$ does not mean the imaginary part is zero,
$i\,0$ means the imaginary part has symmetry like constant 1.
\end{definition}

There are 8 kinds of symmetry property for a real valued function,
there are $8\times 8= 64$ kinds of symmetry property for a complex valued function.
The addition of the value of $Tf$ is taken under the sense of modulation 2.
\begin{lemma}
Given two real symmetric functions $f$ and $g$. If $f+g$ is a symmetric function, then one of the following three conditions must be true:
\begin{itemize}
\item [(i)] $Tf=Tg$;

\item [(ii)] $f=0$;

\item [(iii)] $g=0$.
\end{itemize}
\end{lemma}

The product and convolution of real functions have the following basic properties:
\begin{lemma}
Given $\alpha, \beta\in \{0,1\}^{3}$.
If $Tf(x)= \alpha, Tg(x)= \beta$ and functions $fg$ and $f\ast g$ are well defined, then
$$T(fg)= T(f\ast g)= m(\alpha+\beta).$$
In particular, $\forall t>0$ and $\phi_{t}(x) =t^{-3}\phi(\frac{x}{t})$ satisfying $T\phi(x)= 0$, we have
\begin{equation}
\label{eq:s1}
T(f\ast \phi_t)=Tf.
\end{equation}
\end{lemma}

Since the kernels of $(-\Delta)^{-1}$ and $e^{t\Delta}$ are radial function, we have
\begin{cor}\label{cor:k}
Given $t>0,\alpha, \beta\in \{0,1\}^{3}$.
If $Tf= \alpha+ i\beta $, then
$$T\{(-\Delta)^{-1}f\}=  \alpha+ i\beta.$$
$$T\{e^{t\Delta}f\}=  \alpha+ i\beta.$$
\end{cor}

The derivatives of functions have the following properties:
\begin{lemma} Given $\alpha\in \mathbb{N}, \beta,\gamma \in \{0,1\}^{3}$.
If $Tf(x)= \beta+ i\gamma$, then
$$T(\partial^{\alpha}f) =m (\alpha+\beta)+ i\, m(\alpha+\gamma).$$

\end{lemma}

\begin{proof}
By similarity of proof, we consider only the case $Tf=(0,0,0)+i(0,0,0)$ and $\alpha=(1,0,0)$.
By similarity, assume that $f(x)$ is a real smooth function. Denote $x=(x_1,x')$, we have
$$\begin{array}{rcl}
\partial_1 f(-x_1, x')&=& \lim\limits_{h\rightarrow 0} \frac{f(-x_1+h,x') - f(-x_1, x')}{h}\\
&=& \lim\limits_{h\rightarrow 0} \frac{f(x_1-h,x') - f(x_1, x')}{h}\\
&=& -\partial_1 f(x_1, x').
\end{array}$$
\end{proof}

\section{The kinds of symmetric properties for solenoidal vector fields}
\setcounter{equation}{0}

Many symmetric vector fields have not  divergence zero property, how many symmetric solenoidal vector fields are there?
In this section, I find out all the possible symmetric solenoidal vector field.
For a general vector field $u_0$,
\begin{definition}
We say that a complex vector field $u_0= (u_1, u_2, u_3)^{t} $ has symmetry property, if \,for $l=1,2,3$ fixed,
there exists $\alpha_l, \beta_l\in \{0,1\}^{3}$ such that
\begin{equation}\label{4.1}
Tu_{l}= \alpha_{l}+i \beta_{l}.
\end{equation}
\end{definition}

For each $l$ fixed, there are 8 possibility for $\alpha_l$ and $\beta_l$.
Hence general complex valued vector function has $(8\times 8)^{3}= 262144$ kinds of symmetry property.
Real vector fields has $8^{3}=512$ kinds of symmetry property.
The divergence zero property greatly limits the possible types of symmetries.
I calculate the species by classification.

I consider first
\begin{equation} \label{eq:2.8}
\forall l=1,2,3,
{\mbox { all } }u^{re}_l {\mbox { and }} u^{im}_l {\mbox { are not constant. }}
\end{equation}
We have
\begin{theorem}\label{th:4.1}
Given $u$ satisfies \eqref{4.1} and ${\rm div}\, u=0$. If $u$ satisfies \eqref{eq:2.8},
then there exist $\alpha_0, \beta_0\in \{0,1\}^{3}$ such that
\begin{equation}\label{eq:z:0}
\forall l=1,2,3, Tu_l=m(e_l+ \alpha_0) + i\, m(e_l+ \beta_0).
\end{equation}
\end{theorem}

\begin{proof}
The zero divergence property
${\rm div}\, u= \partial_1 u_1 + \partial_2 u_2 +\partial_3 u_3=0$ implies
$\partial_1 u_1 + \partial_2 u_2$ and $\partial_2 u_2 +\partial_3 u_3$ have still symmetry.
The fact $u$ satisfies \eqref{eq:2.8} implies
\begin{equation} \label{eq:z:r} m(e_1+ \alpha_1)= m( e_2+\alpha_2) = m(e_3+\alpha_3) \end{equation}
\begin{equation} \label{eq:z:i} m(e_1+\beta_1) = m( e_2+\beta_2) = m( e_3+\beta_3).\end{equation}

Denote
\begin{equation} \label{eq:z:r1} \alpha_0= m( e_1+ \alpha_1) =  m( e_2+\alpha_2) =  m( e_3+\alpha_3) \end{equation}
\begin{equation} \label{eq:z:i1} \beta_0=  m( e_1+\beta_1) =  m( e_2+\beta_2) =  m( e_3+\beta_3).\end{equation}
The above equations \eqref{eq:z:r}, \eqref{eq:z:i}, \eqref{eq:z:r1} and \eqref{eq:z:i1} implies that
the symmetry vector field with divergence zero is determined by the symmetry of the first component of vector field.
Hence $u$ satisfies \eqref{eq:z:0}.
\end{proof}

We consider then real valued solenoidal vector fields.
Similar to the above theorem \ref{th:4.1}, 
we have the following possibilities of symmetric properties:
\begin{theorem}\label{7.1}
If $u(t,x)$ is a real symmetric solenoidal vector field, then one of the following three conditions will be satisfied:

(i) $u(t,x)$ has no constant component, then there exists $\alpha_0\in \{0,1\}^{3}$ such that
\begin{equation} \label{real.7.1}
Tu_l= m(e_l +\alpha_0), \forall l=1,2,3.
\end{equation}

(ii) There exists one $\tau \in \{1,2,3\}$ such that $u_{\tau}(t,x)$ is a constant function, then there exists $\alpha_0\in \{0,1\}^{3}$ such that
\begin{equation} \label{real.7.2}
Tu_{\tau}(t,x)=0,  \forall l\neq \tau, Tu_l= m(e_l +\alpha_0).
\end{equation}

(iii) If there exists more than one constant component, then $u$ is a constant vector field and
\begin{equation} \label{real.7.3}
Tu_l= 0, \forall l=1,2,3.
\end{equation}

\end{theorem}
The above theorem says that there exist
only 30 kinds of symmetric solenoidal  vector fields among 512 kinds of real symmetric vector fields.
In fact, we have
\begin{theorem}
There exists only $30$ kinds of real symmetric solenoidal  vector fields.
\end{theorem}
\begin{proof}
We count all the possibility of symmetric cases in Theorem \ref{7.1}.
\begin{itemize}
\item [(i)] There are $8$ kinds of symmetric complex vector fields which satisfy equations \eqref{real.7.1}.

\item [(ii)] In the equations \eqref{real.7.2}, if I fix $\tau$, there exists $8$ kinds of symmetry property.
Further, for each $\tau$ fixed, if $\alpha_0=e_{\tau}$, then it has the similar symmetry property as one of the symmetry property in the equations \eqref{real.7.1}.
There are $21=3\times (8-1)$ kinds of symmetric complex vector fields which are different to which in the equations \eqref{real.7.1}.

\item [(iii)] There are one kind of symmetric complex vector fields in the equations \eqref{real.7.3}.
\end{itemize}

Hence we have $$30=8+ 3\times (8-1) +1.$$
\end{proof}

Thirdly, if there exists constant functions for real or imaginary part, similar to the proof of the above Theorem \ref{th:4.1},
we can prove the following theorem:
\begin{theorem} \label{th:cons}
Given $u$ satisfies \eqref{4.1} and ${\rm div}\, u=0$.

(i) If there are only one component  $\tau$ which is constant function for all the real part function and there is no component which is constant function for imaginary function of $u$,
then there exist
$\alpha_0, \beta_0\in \{0,1\}^{3}$ such that
\begin{equation}\label{eq:z:02}
\begin{array}{rccl}
Tu_l &=& 0 + i m(e_l+ \beta_0), & l=\tau;\\
Tu_l &=& m(e_l+ \alpha_0) + i m(e_l+ \beta_0), & l\neq \tau.
\end{array}
\end{equation}

(ii) If there are only one component  $\tau$ which is constant function for all the imaginary part function and there is no component which constant function for real part function of $u$,
then there exist
$\alpha_0, \beta_0\in \{0,1\}^{3}$ such that
\begin{equation}\label{eq:z:03}
\begin{array}{rccl}
Tu_l &=&  m(e_l+ \alpha_0) + i 0, & l=\tau;\\
Tu_l &=& m(e_l+ \alpha_0) + i m(e_l+ \beta_0), & l\neq \tau.
\end{array}
\end{equation}

(iii) There exists $\tau\in \{1,2,3\}$, the $\tau-$th component of
both real part and imaginary part are constant.
Then there exist $\alpha_0,  \beta_0\in \{0,1\}^{3}$ such that
\begin{equation}\label{eq:z:0x}
\begin{array}{rccl}
Tu_l &=& 0 + i 0, & l=\tau;\\
Tu_l &=& m(e_l+\alpha_0) + i m(e_l+ \beta_0), & l\neq \tau.
\end{array}
\end{equation}

(iv) If there exist $\tau,\tau'\in \{1,2,3\}$ and $\tau\neq \tau'$ such that the $\tau-$th component of
real part and the $\tau'-$th component of imaginary part are constant.
Then there exist $\alpha_0,  \beta_0\in \{0,1\}^{3}$ such that
\begin{equation}\label{eq:z:0y}
\begin{array}{rccl}
Tu_l &=& 0 + i m(e_l+ \beta_0), & l=\tau;\\
Tu_l &=& m(e_l+\alpha_0) + i 0, & l=\tau';\\
Tu_l &=& m(e_l+\alpha_0) + i m(e_l+ \beta_0), & l\neq \tau,\tau'.
\end{array}
\end{equation}

(v) If there are at least two components   which are constant function for all the real part function and there is no component which is constant function for imaginary function of $u$,
then there exist
$ \beta_0\in \{0,1\}^{3}$ such that
\begin{equation}\label{eq:z:04}
\begin{array}{rccl}
Tu_l &=& 0 + i m(e_l+ \beta_0), & l=\tau;\\
Tu_l &=& 0 + i m(e_l+ \beta_0), & l\neq \tau.
\end{array}
\end{equation}

(vi) If there are at least two components   which are constant function for all the imaginary part function and there is no component which constant function for real part function of $u$,
then there exist
$\alpha_0 \in \{0,1\}^{3}$ such that
\begin{equation}\label{eq:z:05}
\begin{array}{rccl}
Tu_l &=&  m(e_l+ \alpha_0) + i 0, & l=\tau;\\
Tu_l &=& m(e_l+ \alpha_0) + i 0, & l\neq \tau.
\end{array}
\end{equation}

(vii) If there are at least two components   which are constant function for both the real part and the imaginary part of $u$,
then such vector fields must be constant such that
\begin{equation}\label{eq:z:06}
\begin{array}{rccl}
Tu_l &=& 0 + i 0, & \forall l=1,2,3.
\end{array}
\end{equation}
\end{theorem}

There are $262144$ kinds of different symmetric complex vector fields, but if they have zero divergence property, then
\begin{theorem}\label{th:5.3}
There are $984$ kinds of different symmetric complex vector fields with zero divergence property.
\end{theorem}
\begin{proof}
We note

\begin{itemize}
\item [(i)] There are $64$ kinds of symmetric complex vector fields which satisfy equations \eqref{eq:z:0}.

\item [(ii)] In the equations \eqref{eq:z:02}, if we fix $\tau$ and $\beta_0$, there exists $8$ kinds of symmetry property.
Further, for each $\tau$ fixed, if $\alpha_0=e_{\tau}$, then it has the similar symmetry property as one of the symmetry property satisfies the equations \eqref{eq:z:0}.
There are $168=3\times 8\times (8-1)$ kinds of symmetric complex vector fields which are different to which in the equations \eqref{eq:z:0}.

\item [(iii)] The same reason as above, there are $168$ kinds of symmetric complex vector fields satisfies the equations \eqref{eq:z:03} which are different to which satisfying the equations \eqref{eq:z:0}.

\item [(iv)] There exist $192=3\times 8\times 8$ kinds of symmetric property in the equations \eqref{eq:z:0x}.
But there exists $\alpha_0=\beta_0=e_{\tau}$ which has the same symmetric property as which in the  equations \eqref{eq:z:0}.
Hence there are $3$ kinds of symmetric properties which have been existed in the equations \eqref{eq:z:0}.

\item [(v)]  There exist $384= 3\times 2\times 8\times 8$ kinds of symmetric property in the equations \eqref{eq:z:0y}.
But for $\alpha_0=e_{\tau}$ and $\beta_0= e_{\tau'}$  which has the same symmetric property as which in the  equations \eqref{eq:z:0}.
Hence there are $6$ kinds of symmetric properties which have been existed in the equations \eqref{eq:z:0}.

\item [(vi)] There are $8$ kinds of symmetric complex vector fields in the equations \eqref{eq:z:04} and \eqref{eq:z:05}.

\item [(vii)] There are one kind of symmetric complex vector fields in the equations \eqref{eq:z:06}.
\end{itemize}
In total, for complex solenoidal vector fields, there are $984 =64+168\times 2+(192-3) + (384-6) + 8\times 2 +1$ kinds of different symmetric properties.

\end{proof}

\section{Non linear term and matched symmetric solenoidal vector fields}
\setcounter{equation}{0}

Most of the time, the Weyl-Helmholtz project operator
$\mathbb{P} \nabla (u \otimes v)$ does not map symmetric solenoidal vector field $u$ and $v$ to symmetric solenoidal vector field.
In fact,
\begin{example}\label{ex:1}
Let $\rho(x), \tilde{\rho}(x)\in C^{\infty}_{0}(\mathbb{R}^{3})$
and $\rho(t,x) = e^{t\Delta}\rho(x), \tilde{\rho}(t,x) = e^{t\Delta} \tilde{\rho}(x)$.
Let $u= u^{re}+ i u^{im}$ and $v= \partial_3 u^{re}+ i u^{im}$ where
$$u^{re} =\left(  \begin{array}{c}
\partial_1 \partial_3\rho(t,x)\\
\partial_2 \partial_3\rho(t,x)\\
-(\partial_1^2 + \partial_2^2) \rho(t,x)
\end{array}
\right)
{\mbox{ and }}
u^{im}= \left(  \begin{array}{c}
\partial_1 \partial_3 \tilde{\rho}(t,x)\\
\partial_2 \partial_3\tilde{\rho}(t,x)\\
-(\partial_1^2 + \partial_2^2) \tilde{\rho}(t,x)
\end{array}
\right)
$$
Then $u$ and $v$ are two symmetric solenoidal vector fields,
but $\mathbb{P} \nabla (u \otimes v)$ and $B(u,v)$ are not symmetric vector fields.
\end{example}
But in this subsection, I can decompose arbitrary solenoidal vector field  $u(t,x)$
into eight kinds of matched symmetric solenoidal vector field $\{u_{\alpha}(t,x)\}_{\alpha\in \{0,1\}}$ such that
$\forall \alpha,\alpha'\in \{0,1\}^{3},
\mathbb{P} \nabla (u_{\alpha}(t,x) \otimes u_{\alpha'}(t,x))$ are always symmetric solenoidal vector field.

We consider first some sufficient conditions for the symmetry of $G(u,v)$ and $B(u,v)$.
\begin{lemma}
Let $u(t,x)=(u_1(t,x), u_2(t,x), u_3(t,x))^{t}$ and $v(t,x)=(v_1(t,x), v_2(t,x), v_3(t,x))^{t}$ be two solenoidal vector fields.
There exists $\alpha_0, \beta_0, \alpha'_0, \beta'_0\in \{0,1\}^{3}$ satisfying
\begin{equation}\label{ab00}
\, m(\alpha_0 + \alpha'_0) = \, m (\beta_0 + \beta'_0 ) \end{equation}
such that $\forall l=1,2,3$, we have
\begin{equation}\label{eq:z:00} \begin{array}{rcl} Tu_l &=& \, m(e_l+ \alpha_0) + i \, m(e_l+ \beta_0);\\
Tv_l &=&\, m (e_l+ \alpha'_0) + i\, m (e_l+ \beta'_0).
\end{array}\end{equation}
Hence we have
\begin{equation}\label{ab0}
TG(u,v)= 0+ i\, m(\alpha_0 + \beta_0).\end{equation}
\end{lemma}
\begin{proof} For $l=1,2,3$, denote the real part  and the imaginary part of $u_l$ to be  $u_l^{re}$ and $u_l^{im}$ respectively;
denote the real part  and the imaginary part of $v_l$ to be  $v_l^{re}$ and $v_l^{im}$ respectively.
Hence, for $l,l'=1,2,3$, the product $u_l v_{l'}$ can be written as
\begin{equation}\label{eq:re.im}
u_l v_{l'}= (u_l^{re} + i u_l^{im}) (v_{l'}^{re}+ i v_{l'}^{im})= u_{l}^{re}v_{l'}^{re} - u_{l}^{im} v_{l'}^{im}
+ i( u_{l}^{im}v_{l'}^{re}  + u_{l}^{re} v_{l'}^{im}).
\end{equation}

Since $u(t,x)$ and $v(t,x)$ satisfies \eqref{eq:z:00}, we have
$$\begin{array}{ccccc}
T(u_{l}^{re}v_{l'}^{re}) &=& \, m(  e_{l}+ \alpha_0 + e_{l'}+ \alpha'_0) &=& \, m( e_{l} + e_{l'}+ \alpha_0 + \alpha'_0) ,\\
T(u_{l}^{im} v_{l'}^{im})&= &\, m( e_{l} + \beta_0 + e_{l'}+ \beta'_0) &= & \, m( e_{l} + e_{l'}+ \beta_0 + \beta'_0)\\
T(u_{l}^{im}v_{l'}^{re}) &=& \, m( e_{l}+ \beta_0 + e_{l'}+ \alpha'_0) &=& \, m( e_{l} + e_{l'}+ \alpha'_0 + \beta_0) ,\\
T(u_{l}^{re} v_{l'}^{im})&= & \, m( e_{l} + \alpha_0 + e_{l'}+ \beta'_0) &= & \, m( e_{l} + e_{l'}+ \alpha_0 + \beta'_0) .
\end{array}$$
Hence \eqref{ab00} implies
\begin{equation}\label{ll}Tu_l v_{l'}= m (e_{l} + e_{l'}+ \alpha_0 + \alpha'_0) + i\, m(e_{l} + e_{l'}+ \alpha'_0 + \beta_0).\end{equation}
The above \eqref{ll} implies,
$$T\{\sum\limits_{l,l'} \partial_l\partial_{l'} (u_l v_{l'})\}= \, m( \alpha_0 + \alpha'_0)+ i\, m(\alpha'_0 + \beta_0).$$
\end{proof}

\begin{lemma}\label{le:4.3}
If $u(t,x)$ and $v(t,x)$ satisfies \eqref{ab00} and \eqref{eq:z:00}, then
\begin{equation} \label{XB}
T(B(u,v)) =
\left( \begin{aligned}
\, m(e_1+ \alpha_0 + \alpha'_0) + i\, m(e_1+ \alpha'_0 + \beta_0) \\
\, m(e_2 + \alpha_0 + \alpha'_0)+ i\, m(e_2+ \alpha'_0 + \beta_0)\\
\, m(e_3 + \alpha_0 + \alpha'_0) + i\, m(e_3+ \alpha'_0 + \beta_0)
 \end{aligned}\right ),
\end{equation}
\end{lemma}

\begin{proof}
For $l'=1,2,3$, according to \eqref{ab0}, we have
$$T\{\partial _{l'}\sum\limits_{l,l''} \partial_l\partial_{l''} (u_l v_{l''})\}= \, m( e_{l'}+ \alpha_0 + \alpha'_0) + i\, m(e_{l'}+\alpha'_0 + \beta_0).$$
According to \eqref{ll}, we have
$$T\{\sum_{l}  \partial _{l}(u_{l}v_{l'})\}= \, m( e_{l'}+ \alpha_0 + \alpha'_0) + i\, m( e_{l'}+ \alpha'_0 + \beta_0).$$

Hence
\begin{equation*}
T(\mathbb{P} \nabla (u\otimes v)) =
\left( \begin{aligned}
\, m(e_1+ \alpha_0 + \alpha'_0) + i\, m(e_1+ \alpha'_0 + \beta_0) \\
\, m(e_2+ \alpha_0 + \alpha'_0)+ i\, m(e_2+ \alpha'_0 + \beta_0)\\
\, m(e_3+ \alpha_0 + \alpha'_0)+ i\, m(e_3+ \alpha'_0 + \beta_0)
 \end{aligned}\right ).
\end{equation*}
By corollary \ref{cor:k},
$T(B(u,v))$ satisfies the equation \eqref{XB}.
\end{proof}

We introduce then a definition of matched symmetric property to consider the necessary and sufficient conditio for the symmetric property of $B(u,v)$:
\begin{definition}
Given two solenoidal vector fields $u= (u^{re}_1+i u^{im}_1, u^{re}_2+ iu^{im}_2 , u^{re}_3+i u^{im}_3)^{t} $ and
$v= (v^{re}_1+i v^{im}_1, v^{re}_2+ iv^{im}_2 , v^{re}_3+i v^{im}_3)^{t} $. Say $u$ and $v$ have matched symmetric properties,
if $\forall \alpha\in \{0,1\}^3$ and $\tau=1,2,3$, there exist functions $f^{re}_{\tau}(\alpha)$, $f^{im}_{\tau}(\alpha)$, $g^{re}_{\tau}(\alpha)$ and $g^{im}_{\tau}(\alpha)$ belong to $\{0,1\}$ such that all $(u^{re}_{\tau}, f^{re}_{\tau})$, $(u^{im}_{\tau}, f^{im}_{\tau})$,
$(v^{re}_{\tau}, g^{re}_{\tau})$ and $(v^{im}_{\tau}, g^{im}_{\tau})$ satisfy \eqref{eq:sr1} and the following matched condition
\begin{equation}\label{eq£»match}
m(f^{re}_{\tau}+ g^{re}_{\tau})= m(f^{im}_{\tau}+ g^{im}_{\tau}), \forall \tau=1,2,3.
\end{equation}
\end{definition}
The matched symmetric property \eqref{eq£»match} is just
the condition \eqref{ab00}.

Let $u(t,x)=(u_1(t,x), u_2(t,x), u_3(t,x))^{t}$ and $v(t,x)=(v_1(t,x), v_2(t,x), v_3(t,x))^{t}$ be two solenoidal vector fields.
Let $u^{re}$ and $v^{re}$ be the relative real part and $u^{im}$ and $v^{im}$ be the imaginary part.

\begin{theorem}\label{th:4.5.11}
Let $u,v$ be two solenoidal vector fields such that
there exists $\alpha_0, \beta_0, \alpha'_0, \beta'_0\in \{0,1\}^{3}$ satisfying that $\forall l=1,2,3$,
$u_l$ and $v_l$ satisfy the equation \eqref{eq:z:00}.
Then $B(u,v)$ is a symmetric vector field iff one of the following facts is true:
\begin{equation*} 
{\mbox { Equation \eqref{ab00} is true }}. \end{equation*}
\begin{equation}\label{ab01}
\, \mathbb{P} \nabla (u, v^{re})=0. \end{equation}
\begin{equation}\label{ab02}
\, \mathbb{P} \nabla (u, v^{im})=0. \end{equation}
\begin{equation}\label{ab03}
\, \mathbb{P} \nabla (u^{re}, v)=0. \end{equation}
\begin{equation}\label{ab04}
\, \mathbb{P} \nabla (u^{im}, v)=0. \end{equation}
\end{theorem}

\begin{proof}
According to \eqref{eq:re.im}, we have
\begin{equation}\label{eqn:c}
C(u,v)= \mathbb{P}\nabla (u,v)= \mathbb{P}\nabla (u^{re},v^{re}) - \mathbb{P}\nabla (u^{im},v^{im}) + i( \mathbb{P}\nabla (u^{im},v^{re})- \mathbb{P}\nabla (u^{re},v^{im})).
\end{equation}
Applying \eqref{eq:z:00}, we have
$$\begin{array}{ccccc}
T(u_{l}^{re}v_{l'}^{re}) &=& \, m(  e_{l}+ \alpha_0 + e_{l'}+ \alpha'_0) &=& \, m( e_{l} + e_{l'}+ \alpha_0 + \alpha'_0) ,\\
T(u_{l}^{im} v_{l'}^{im})&= &\, m( e_{l} + \beta_0 + e_{l'}+ \beta'_0) &= & \, m( e_{l} + e_{l'}+ \beta_0 + \beta'_0).
\end{array}$$
Hence
$$\begin{array}{ccc}
TC_{l}(u^{re}, v^{re})  &=& \, m( e_{l} + \alpha_0 + \alpha'_0) ,\\
TC_{l}(u^{im}, v^{im}) &= & \, m( e_{l} + \beta_0 + \beta'_0).
\end{array}$$
According to \eqref{eqn:c}, $C^{re}_{l}(u,v)$ has symmetry iff one of the following three conditions is true:
\begin{equation*}
{\mbox { Equation \eqref{ab00} is true }}. \end{equation*}
\begin{equation}\label{ab01}
\, \mathbb{P} \nabla (u^{re}, v^{re})=0. \end{equation}
\begin{equation}\label{ab02}
\, \mathbb{P} \nabla (u^{im}, v^{im})=0. \end{equation}

By \eqref{eq:z:00}, we have
$$\begin{array}{ccccc}
T(u_{l}^{im}v_{l'}^{re}) &=& \, m( e_{l}+ \beta_0 + e_{l'}+ \alpha'_0) &=& \, m( e_{l} + e_{l'}+ \alpha'_0 + \beta_0) ,\\
T(u_{l}^{re} v_{l'}^{im})&= & \, m( e_{l} + \alpha_0 + e_{l'}+ \beta'_0) &= & \, m( e_{l} + e_{l'}+ \alpha_0 + \beta'_0) .
\end{array}$$
Hence
$$\begin{array}{ccc}
TC_{l}(u^{im}, v^{re})  &=& \, m( e_{l} + \alpha'_0+ \beta_0) ,\\
TC_{l}(u^{re}, v^{im}) &= & \, m( e_{l} + \alpha_0 + \beta'_0).
\end{array}$$
According to \eqref{eqn:c}, $C^{im}_{l}(u,v)$ has symmetry iff one of the following three conditions is true:
\begin{equation*} 
{\mbox { Equation \eqref{ab00} is true }}. \end{equation*}
\begin{equation}\label{ab011}
\, \mathbb{P} \nabla (u^{im}, v^{re})=0. \end{equation}
\begin{equation}\label{ab012}
\, \mathbb{P} \nabla (u^{re}, v^{im})=0. \end{equation}

Combine with equations \eqref{ab00}, \eqref{ab01}, \eqref{ab02}, \eqref{ab011}, \eqref{ab012}, we get this Theorem.

\end{proof}

Example \ref{ex:1} tells us that,
if two symmetric solenoidal vector fields $u$ and $v$ are not matched,
then $B(u,v)$ may not be symmetric. But I can prove that
any solenoidal vector field can be decomposed to eight matched symmetric solenoidal vector fields.
\begin{theorem}\label{th:4.6.11}
Given any solenoidal vector field $u(t,x)=(u_1(t,x), u_2(t,x), u_3(t,x))^{t}$ and $\beta\in \{0,1\}^{3}$.
Then there exist 8 symmetric solenoidal vector field
$u_{\alpha}= u^{\alpha,re}+ i u^{\alpha,im}$ where $u^{\alpha,re}$ and $ u^{\alpha,im}$ are respectively the real part and imaginary part of $u_{\alpha}$
such that $u= \sum\limits_{\alpha\in \{0,1\}^{3}} u_{\alpha}$ satisfying the following condition
$$T (u^{\alpha,re}_l u^{\alpha,im}_l) = \beta, \forall l=1,2,3.$$
Hence $\forall \alpha \neq \alpha'\in \{0,1\}^{3}$, we have $u^{\alpha}$ and $ u^{\alpha'}$ has matched property.
\end{theorem}

\begin{proof}
We decompose first the real part $u^{re}$. For $l=1,2,3$, each $u_{l}^{re}$ can be decomposed as 8 symmetric functions $\{u_{l}^{\alpha, re}\}_{\alpha\in \{0,1\}^{3}}$.
$\forall \alpha\in \{0,1\}^{3}$ and $l=1,2,3$, I take $u_{l}^{\alpha, re}$ such that $Tu_{l}^{\alpha, re}= e_{l}+\alpha.$
It is easy to see that
\begin{equation} \label{eq:re:c.s}
T\sum\limits_{l=1,2,3} \partial_{l} u_{l}^{\alpha, re}= \alpha, \forall \alpha\in \{0,1\}^{3}.
\end{equation}
Now I prove that
$$\sum\limits_{l=1,2,3} \partial_{l} u_{l}^{\alpha, re}=0.$$
Denote $\tilde{u}_{l}^{\alpha,re}= u_{l}^{re} - \sum\limits_{\tau\neq \alpha} u_{l}^{\tau, re}.$
Since $\sum\limits_{l=1,2,3} \partial_{l} u_{l}^{ re}=0,$ by \eqref{eq:re:c.s},
we have $$\alpha=T\sum\limits_{l=1,2,3} \partial_{l} u_{l}^{\alpha, re}= T\sum\limits_{l=1,2,3} \partial_{l} \tilde{u}_{l}^{\alpha, re}.$$
But if $\sum\limits_{l=1,2,3} \partial_{l} \tilde{u}_{l}^{\alpha, re}\neq 0$, then $T\sum\limits_{l=1,2,3} \partial_{l} \tilde{u}_{l}^{\alpha, re}\neq \alpha.$
Hence we must have $$\sum\limits_{l=1,2,3} \partial_{l} u_{l}^{\alpha, re}=\sum\limits_{l=1,2,3} \partial_{l} \tilde{u}_{l}^{\alpha, re}=0.$$

For the imaginary part,  repeat almost the same process with little modifications.
We decompose the imaginary part $u^{im}$. For $l=1,2,3$, each $u_{l}^{im}$ can be decomposed as 8 symmetric functions $u_{l}^{\alpha, im}$.
$\forall \alpha\in \{0,1\}^{3}$ and $l=1,2,3$, we take $u_{l}^{\alpha, im}$ such that $Tu_{l}^{\alpha, im}= e_{l}+\alpha+\beta.$
It is easy to see that
\begin{equation} \label{eq:im:c.s}
T\sum\limits_{l=1,2,3} \partial_{l} u_{l}^{\alpha, im}= \alpha+\beta, \forall \alpha\in \{0,1\}^{3}.
\end{equation}
The same reason as above for real part, we have
$$\sum\limits_{l=1,2,3} \partial_{l} u_{l}^{\alpha, im}=0.$$
By applying \eqref{eq:re:c.s} and \eqref{eq:im:c.s}, we get
$$T (u^{\alpha,re}_l u^{\alpha,im}_l) = \beta, \forall l=1,2,3.$$
We get the relative decomposition.
\end{proof}

For two matched symmetric solenoidal vector field,
we prove the bilinear operator $B(u,v)$ maps always symmetric solenoidal vector fields to symmetric solenoidal vector field.
Let $B(u,v)= (B^{re}_1(u,v)+i B^{im}_1(u,v), B^{re}_2(u,v)+ iB^{im}_2(u,v) , B^{re}_3(u,v)+i B^{im}_3(u,v))^{t} $.
Similar to the proof of Theorem \ref{th:4.6.11}, we have

\begin{theorem} \label{th.1.8}
Let $u= (u^{re}_1+i u^{im}_1, u^{re}_2+ iu^{im}_2 , u^{re}_3+i u^{im}_3)^{t} $ and
$v= (v^{re}_1+i v^{im}_1, v^{re}_2+ iv^{im}_2 , v^{re}_3+i v^{im}_3)^{t} $ be solenoidal vector field with matched symmetric properties.
That is to say, $\forall \alpha\in \{0,1\}^3$ and $\tau=1,2,3$, there exist functions $f^{re}_{\tau}(\alpha)$,
$f^{im}_{\tau}(\alpha)$, $g^{re}_{\tau}(\alpha)$ and $g^{im}_{\tau}(\alpha)$ belong to $\{0,1\}$ such that all $(u^{re}_{\tau}, f^{re}_{\tau})$, $(u^{im}_{\tau}, f^{im}_{\tau})$,
$(v^{re}_{\tau}, g^{re}_{\tau})$ and $(v^{im}_{\tau}, g^{im}_{\tau})$ satisfy \eqref{eq:sr1} and satisfies \eqref{eq£»match}.
Then $B(u,v)$ has symmetry property and
all $(B^{re}_{\tau}(u,v), f^{re}_{\tau}+ g^{re}_{\tau})$, $(B^{im}_{\tau}(u,v), f^{im}_{\tau}+ g^{im}_{\tau})$ satisfy \eqref{eq:sr1}.
\end{theorem}

Theorems \ref{th:4.5.11} and \ref{th.1.8} tell us that,
as long as we have known the structure of the nonlinear quantities of all the symmetric helical vector field,
the structure of the nonlinear quantities of any helical vector field is clear.
\begin{theorem}\label{th.8.9}
Given any two solenoidal vector fields $u(t,x)$ and $v(t,x)$. Then there exists symmetric solenoidal vector fields $u^{\alpha}$ and $v^{\alpha}$ satisfying
$u=\sum\limits_{\alpha\in \{0,1\}^{3}} u^{\alpha}$ and
$v=\sum\limits_{\alpha'\in \{0,1\}^{3}} v^{\alpha'}$ such that
all the $B(u^{\alpha}, v^{\alpha'})$ have symmetric properties satisfying
$$B(u,v)=\sum\limits_{\alpha,\alpha'\in \{0,1\}^{3}} B(u^{\alpha}, v^{\alpha'}).$$
\end{theorem}

\begin{remark}
The above Theorem \ref{th.8.9}
told us, to understand the structure of the non-linear terms,
we need only to know the structure  for symmetric initial data in the sense of harmonic analysis.
\end{remark}

\section{Rigidity for real initial values}
\setcounter{equation}{0}

Leray guessed that, a blow-up solution should have similar structure as its initial data and proposed to consider self-similar solution.
But Necas-Ruzicka-Sverak \cite{NRS} proved in 1996  that such solution should be zero.
That is to say, Navier-Stokes equations have rigidity for self-similar structure.
When Yang-Yang-Wu study the illposedness of \eqref{1.1} in \cite{YYW},
they found symmetry of independent variables can superimposed effect.
A natural problem is to consider the rigidity and symmetry.
The existence of Beltrami flow makes the situation different to the self-simalar cases.
Constantin-Majda \cite{CM} and Lei-Lin-Zhou \cite{LLZ} have constructed some vector fields $u_0$ and Beltrami flows.
For example,
\begin{example}\label{ex:2}
Let $\rho (r) \in C^{3} ([0, \infty))$. $u_1=\partial_2 \rho (\sqrt{x_1^2 + x_2^2})$,
$u_2= -\partial_1 \rho (\sqrt{x_1^2 + x_2^2})$ and $u_3=0$.
Then non-constant $u_0= (u_1, u_2, u_3)^{t}$ is a symmetric solenoidal vector field.

The equation \eqref{1.n} is very complex.
But the above $u_0$ satisfies the equations \eqref{1.n}
and  $u(t,x)=e^{t\Delta}u_0$ is the solution of the equations \eqref{1.1}.
Further, With the right $\rho$, the norm $\|u_0\|_{(\dot{B}^{-1,\infty}_{\infty}(\mathbb{R}^{3}))^{3}}$ can be enough big.
\end{example}
In this section, I prove, except Beltrami flow, there exists only one kind of symmetric property which allow real symmetric solution for equations \eqref{1.1}.
At the begin of this section, I prove first
the condition \eqref{real.7.2} in the above theorem \ref{7.1} will imply the condition \eqref{1.n}.
To prove this fact,
we will distinguish two cases for the condition \eqref{real.7.2}
in the following two theorems.
\begin{theorem}\label{7.2}
Given $u$ satisfies \eqref{real.7.2} and $\alpha_0\neq e_{\tau}$ and $\alpha_0\neq 0$. Then
$u$ has the same symmetry property for all $t\geq 0$  if and only if \eqref{1.n} is true.
\end{theorem}

\begin{proof}

Given $u$ satisfies \eqref{real.7.2}, then $A_{\tau}(u,u)= 0$ and
$$\begin{array}{rcl}
G(u,u)&=&  \sum\limits_{|l-\tau| |l'-\tau|\neq 0} \partial_{l}\partial_{l'} ( u_{l} u_{l'}).
\end{array}$$
Hence, $TG(u,u)= 0$.

Then I consider two cases. (i) If $u_{\tau}\neq 0$, then $A_{\tau}(u,u)= 0$ and $u_{0,\tau}$ is a constant, and
$$\begin{array}{rcl}
G(u,u)&=&  \sum\limits_{|l-\tau| |l'-\tau|\neq 0} \partial_{l}\partial_{l'} ( u_{l} u_{l'}).
\end{array}$$
Hence, $TG(u,u)= 0$  and $T \partial_{\tau} G(u,u)= \alpha_0.$
If $u_{\tau}(t,x)$ satisfies \eqref{1.3},
then $$C_{\tau}(u,u)=0.$$

Let $l\neq \tau, l'\neq l,\tau$. We have
$$\begin{array}{rcl}A_{l}(u,u) &=& \{\partial_{l}(u_{l})^{2} +
\partial_{l'}(u_{l'} u_{l})\} +\partial_{\tau}(u_{\tau} u_{l}) \equiv I+II.
\end{array}$$
It is easy to see $TI= e_{l}$ and $TII= \alpha_0$.
If $u_{l}(t,x)$ satisfies \eqref{1.3}, then $I=0, II=0.$
But $II=0$ implies $u_{\tau} u_{l}$ is a constant.
That is to say, all $u_{l},l=1,2,3$ are constant and \eqref{1.n} is true.

(ii) If $u_{\tau}=0$, then $A_{\tau}(u,u)= 0$, $u_{0,\tau}=0$  and further
$$\begin{array}{rcl}
G(u,u)&=&  \sum\limits_{|l-\tau| |l'-\tau|\neq 0} \partial_{l}\partial_{l'} ( u_{l} u_{l'}).
\end{array}$$
Hence, $TG(u,u)= 0$ and $T\partial_{l} G(u,u) = e_{l}, \forall l=1,2,3 $.
Therefore $C_{\tau}(u,u)$ has symmetry property and $TC_{\tau}(u,u)= e_{\tau} $.
If $Tu_{\tau}= TB_{\tau}(u,u)$,  by \eqref{1.3}, we have
$C_{\tau} (u)= 0$. Which implies $\partial_{\tau} G(u,u)=0$, hence $G(u,u)$ is a constant.
So for $l\neq \tau$, we have $C_{l}(u,u)=A_{l}(u,u)$.

Let $l\neq \tau, l'\neq l,\tau$. We have
$$\begin{array}{rcl}A_{l}(u,u) &=& \partial_{l}(u_{l})^{2} +
\partial_{l'}(u_{l'} u_{l}).
\end{array}$$
Then we have $TA_{l}(u,u)= e_{l}$.
Since $\alpha_{0}\neq 0$, we have must $A_{l}(u,u)=0$.
That is to say, \eqref{1.n} is true.

\end{proof}

\begin{theorem}\label{7.3}
Given $u$ satisfies \eqref{real.7.2}, $\alpha_0=0 $ and $u_{\tau}\neq 0$. Then
$u$ has the same symmetry property for all $t\geq 0$  if and only if \eqref{1.n} is true.
\end{theorem}

\begin{proof}

Given $u$ satisfies \eqref{real.7.2}, then $A_{\tau}(u,u)= 0$ and
$$\begin{array}{rcl}
G(u,u)&=&  \sum\limits_{|l-\tau| |l'-\tau|\neq 0} \partial_{l}\partial_{l'} ( u_{l} u_{l'}).
\end{array}$$
Hence, $TG(u,u)= 0$.

If $u_{\tau}\neq 0$, then $A_{\tau}(u,u)= 0$ and $u_{0,\tau}$ is a constant, and
$$\begin{array}{rcl}
G(u,u)&=&  \sum\limits_{|l-\tau| |l'-\tau|\neq 0} \partial_{l}\partial_{l'} ( u_{l} u_{l'}).
\end{array}$$
Hence, $TG(u,u)= 0$  and $T \partial_{\tau} G(u,u)= \alpha_0$.
If $u_{\tau}(t,x)$ satisfies \eqref{1.3}, 
then $$C_{\tau}(u,u)=0.$$

Let $l\neq \tau, l'\neq l,\tau$. We have
$$\begin{array}{rcl}A_{l}(u,u) &=& \{\partial_{l}(u_{l})^{2} +
\partial_{l'}(u_{l'} u_{l})\} +\partial_{\tau}(u_{\tau} u_{l}) \equiv I+II.
\end{array}$$
It is easy to see
$TI= e_{l}$ and $TII= e_{\tau}+ e_{l}$.

If $u_{l}(t,x)$ satisfies \eqref{1.3}, then $II=0.$  But $II=0$ implies $u_{\tau} u_{l}$ is a constant.
That is to say, all $u_{l},l=1,2,3$ are constant and \eqref{1.n} is true.

\end{proof}

If $u$ is real, the non-linear term $B(u,u)$ have
$3\times (3+6)=27$ terms which include products of differentiable functions, integration,
$e^{(t-s)\Delta}$ and $(-\Delta)^{-1}$.
{\bf There are many factors which can change the symmetry property of $B(u,u)(t,x)$},
hence it is not so easy to visualize that one can control the symmetry property of the solution.
In this section,
we find out all the possible cases which can generate real symmetric solutions.
For real valued initial data, we have:
\begin{theorem} \label{th:main0}
Let $u(t,x)$ be the real valued solution of the equations \eqref{1.1}.
$u(t,x)$ has the same symmetric property for all $t\geq 0$, if and only if one of the following conditions holds:

(i) the real valued solenoidal vector field $u_0(x)$ satisfies the following condition \eqref{eq:f.a}
\begin{equation} \label{eq:f.a}
T(u_{0})=
\left( \begin{aligned}
e_1  \\
e_2 \\
e_3
 \end{aligned}\right ).
\end{equation}

(ii)  the real valued solenoidal vector field $u_0(x)$ satisfies the following condition \eqref{1.n}

\end{theorem}
The above theorem tells us, if $u_0$ does not satisfy the symmetric property \eqref{eq:f.a}, then $u_0$ must be a solution of eqation \eqref{1.n}.
That is to say, the solution $u(t,x)$ has almost rigidity.
For real valued initial data, except for Beltrami flow,
there exists only one kind of symmetric property for \eqref{eq:f.a}.
Now we come to prove theorem \ref{th:main0}.
\begin{proof}
We can classify the particular case in \eqref{real.7.2} where $\alpha_0=0$ and $u_{\tau}=0$ to the case \eqref{real.7.1} with $\alpha_0=0$.
By Theorems \ref{7.2} and \ref{7.3}, we need only to consider the case where $u$ satisfies \eqref{real.7.1}.
Similar to lemma \ref{le:44.3}, we have
\begin{equation} \label{12.4}
T(B(u,u)) =
\left( \begin{aligned}
e_1  \\
e_2 \\
e_3
 \end{aligned}\right ).
\end{equation}

$u$ has the same symmetric property for all $t\geq 0$, we have
$$Tu=Tu_{0}= Te^{t\Delta} u_{0}.$$
According to lemma \ref{1.3}, we have
$$ B(u,u)= e^{t\Delta}u_{0} - u.$$
Combine with \eqref{12.4}, we get
$\alpha_0=0$ in the equation \eqref{real.7.1} or $B(u,u)=0$.

\begin{itemize}
\item [(i)]
$\alpha_0=0$ in the equation \eqref{real.7.1} implies $u_0$ satisfies
equation \eqref{eq:f.a}

\item[(ii)]
$B(u,u)=0$ implies that $u_0$ satisfies \eqref{1.n}.
\end{itemize}
If $u_0$ is a real initial data and satisfies (i) or (ii), then $u^{\tau}(t,x)$ in equation \eqref{1.4} are real symmetric functions.
Hence $u(t,x)$ is a real symmetric solution.

\end{proof}

We remark two points here on rigidity and
the iterative algorithm \eqref{1.4}.
\begin{remark}
(i) $u(t,x)$ in example \ref{ex:2} does not satisfy the condition \eqref{eq:f.a} and
\begin{equation*}
T(u(t,x)) = T(u_{0})=
\left( \begin{aligned}
e_2 \,\,\,\,\,\,\,\,\,\,\,\, \\
e_1 \,\,\,\,\,\,\,\,\,\,\,\, \\
e_1+e_2+e_3
\end{aligned}\right ).
\end{equation*}
but satisfies \eqref{1.n}.

(ii) Theorem \ref{th:main0} tells us that,
$u^{\tau}(t,x)$ in \eqref{1.4} have the same symmetric property for all $\tau\geq 0$ and $t\geq 0$
satisfying $u^{\tau}_{1}(t,0,x_2,x_3)= u^{\tau}_{2}(t,x_1,0,x_3)= u^{\tau}_{3}(t,x_1,x_2,0)=0$ and
{\bf the limitation of $u^{\tau}(t,x)$ have no blow-up phenomenon on certain coordinate axis.}
\end{remark}

Here's an opening question:
\begin{remark}
Can we solve the equations \eqref{1.n} completely by following two steps?

(i) Which symmetric functions satisfy the equations \eqref{1.n}?

(ii) The case of asymmetric functions is solved in combination
with the result of symmetry function and theorem \ref{th:4.6.11}.
\end{remark}

\section{Rigidity for symmetric complex solenoidal vector field}
\setcounter{equation}{0}
For symmetric complex solenoidal vector fields,
we can establish the similar rigidity results as real vector fields.
But if $u$ is complex, (i) the non-linear term $B(u,u)$ has more terms.
$B(u,u)$ have $3\times (11+12+9)=96$ terms
which include products of differentiable functions, integration, operators $e^{(t-s)\Delta}$ and $(-\Delta)^{-1}$.
(ii) there exist more kinds of symmetric property for complex cases.
To save the space of the article,
in this section, I consider only the rigidity for symmetric complex solenoidal vector field
satisfying \eqref{eq:2.8}.

We consider first the symmetry of $G(u,u)$.
\begin{lemma}
If $u(x)$ satisfies \eqref{eq:z:0}, then
\begin{equation}\label{ab000}
TG(u,u)= 0+ i \,m(\alpha_0 + \beta_0).\end{equation}
\end{lemma}
\begin{proof} For $l=1,2,3$, denote the real part  and the imaginary part of $u_l$ to be  $u_l^{re}$ and $u_l^{im}$ respectively.
Hence, for $l,l'=1,2,3$, the product $u_l u_{l'}$ can be written as
$$u_l u_{l'}= (u_l^{re} + i u_l^{im}) (u_{l'}^{re}+ i u_{l'}^{im})= u_{l}^{re}u_{l'}^{re} - u_{l}^{im} u_{l'}^{im}
+ i( u_{l}^{im}u_{l'}^{re}  + u_{l}^{re} u_{l'}^{im})$$

Since $u(x)$ satisfies \eqref{eq:z:0}, we have
$$T(u_{l}^{re}u_{l'}^{re} - u_{l}^{im} u_{l'}^{im})= m (e_l+ e_{l'}),$$
and
$$\begin{array}{ccccc}T(u_{l}^{im}u_{l'}^{re}) &=& m( e_{l}+ \beta_0 + e_{l'}+ \alpha_0) &=& m( e_{l} + e_{l'}+ \alpha_0 + \beta_0) ,\\
T(u_{l}^{re} u_{l'}^{im})&= & m( e_{l} + \alpha_0 + e_{l'}+ \beta_0) &= & m( e_{l} + e_{l'}+ \alpha_0 + \beta_0) .\end{array}$$
Hence
\begin{equation}\label{ll8}Tu_l u_{l'}= m(e_{l} + e_{l'}) + im(e_{l} + e_{l'}+ \alpha_0 + \beta_0).\end{equation}
The above \eqref{ll8} implies,
$$T\{\sum\limits_{l,l'} \partial_l\partial_{l'} (u_l u_{l'})\}= 0+ i\, m(\alpha_0 + \beta_0).$$
\end{proof}

Then we consider the symmetry of $B(u,u)$.
\begin{lemma}\label{le:44.3}
If $u(x)$ satisfies \eqref{eq:z:0}, then
\begin{equation} \label{XB000}
T(B(u,u)) =
\left( \begin{aligned}
e_1 + i\, m(e_1+ \alpha_0 + \beta_0) \\
e_2 + i\, m(e_2+ \alpha_0 + \beta_0)\\
e_3 + i\, m (e_3+ \alpha_0 + \beta_0)
 \end{aligned}\right ),
\end{equation}
\end{lemma}

\begin{proof}
For $l'=1,2,3$, according to \eqref{ab000}, we have
$$T\{\partial _{l'}\sum\limits_{l,l''} \partial_l\partial_{l''} (u_l u_{l''})\}= e_{l'} + i\, m(e_{l'}+\alpha_0 + \beta_0).$$
According to \eqref{ll8}, we have
$$T\{\sum_{l}  \partial _{l}(u_{l}u_{l'})\}=  e_{l'} + i\, m( e_{l'}+ \alpha_0 + \beta_0).$$
Hence
\begin{equation*}
T(\mathbb{P} \nabla (u\otimes u)) =
\left( \begin{aligned}
e_1 + i\, m(e_1+ \alpha_0 + \beta_0) \\
e_2 + i\, m(e_2+ \alpha_0 + \beta_0)\\
e_3 + i\, m(e_3+ \alpha_0 + \beta_0)
 \end{aligned}\right ).
\end{equation*}
By corollary \ref{cor:k},
$T(B(u,u))$ satisfies the equation \eqref{XB000}.
\end{proof}

Let $u(t,x)=(u^{re}_1+i u^{im}_1, u^{re}_2+ iu^{im}_2 , u^{re}_3+i u^{im}_3)^{t}$ be the strong solution of \eqref{1.1}
satisfying \eqref{eq:2.8}.
For complex valued initial data, to save the length of the paper, I assume that $u(t,x)$ satisfies \eqref{eq:2.8}.
Under this constraint, {\bf there are 8 kinds and only 8 kinds of symmetry property which can generate symmetric solution}.
Denote $e_{1}=(1,0,0), e_{2}=(0,1,0)$ and $e_3=(0,0,1)$. We have
\begin{theorem} \label{th:main1}
Let $u(t,x)$ be the strong solution of \eqref{1.1} satisfying \eqref{eq:2.8}.
$u(t,x)$ has the same symmetric property for all $t\geq 0$ if and only if one of the following conditions is satisfied
\begin{itemize}
\item [(i)] there exists $\beta_{0}\in \{0,1\}^{3}$ such that
\begin{equation} \label{eq:g.b}
T(u_{0})= 
\left( \begin{aligned}
e_1 + i\, m(e_1+ \beta_0) \\
e_2 + i\, m(e_2+ \beta_0)\\
e_3 + i\, m(e_3+ \beta_0)
 \end{aligned}\right ).
\end{equation}
\item [(ii)] $u_{0}(x)$ satisfies the condition \eqref{1.n}.
\end{itemize}
\end{theorem}

\begin{proof}

$u(t,x)$ has the same symmetric property for all $t\geq 0$,
$$Tu(t,x)= Tu_0= T e^{t\Delta} u_0.$$
Further, if $u(t,x)$ has symmetric property and satisfies \eqref{eq:2.8},
then there exists $\alpha_0, \beta_0 \in \{0,1\}^{3}$ such that $u(t,x)$ satisfies
\eqref{eq:z:0}.
According to lemma \ref{le:44.3}, we have
\begin{equation} \label{12.3}
T(B(u,u)) =
\left( \begin{aligned}
e_1 + i\, m(e_1+ \alpha_0 + \beta_0) \\
e_2 + i\, m(e_2+ \alpha_0 + \beta_0)\\
e_3 + i\, m(e_3+ \alpha_0 + \beta_0)
 \end{aligned}\right ),
\end{equation}

According to equation \eqref{1.3}, we have
\begin{equation}\label{12.1}
e^{t\Delta}u_{0}(x)-u(t, x)= B(u, u)(t, x).
\end{equation}
Combine \eqref{eq:2.8}, \eqref{12.3} and \eqref{12.1},  we get
$$\alpha_0=0 \mbox{ or } B(u,u)=0.$$

\begin{itemize}
\item [(i)] $\alpha_0=0$ implies $u_0$ satisfies equation \eqref{eq:g.b}.

\item [(ii)] $B(u,u)=0$ implies $u= e^{t\Delta}u_0$ and $\partial_t u -\Delta u=0$. That is to say, \eqref{1.n} is right.
\end{itemize}
If $u_0$ satisfies the above (i) or (ii), then $u^{\tau}(t,x)$ in equation \eqref{1.4} are symmetric functions.
Hence $u(t,x)$ is a symmetric solution.

\end{proof}

\section{Smooth solution with energy conservation}
\setcounter{equation}{0}

As an application of symmetry, we study Navier-Stokes equations on the domain $\Omega$.
We consider the following incompressible Navier-Stokes equations on the space $\mathbb{R}_{+}\times \Omega$,
\begin{equation}\label{11.2}
\left\{ \begin{aligned}
  &\partial_t u -\Delta u+u\cdot\nabla u-\nabla p=0,\;\;(t,x) \in \mathbb{R}_{+}\times \Omega, \\
  &\nabla\cdot u=0,\\
  u&(0, x)=u_{0}(x),
 \end{aligned}\right.
 \end{equation}
where initial data $u_{0}(x)$ is real valued.

When considering Navier-Stokes equations on the domain, usually, one makes first zero extension of initial data,
then considers well-posedness on the whole spaces. The solution of the domain is the restriction of the solution on the whole space to the domain.
There exists loss of energy for such extension.
Given $\Omega= \mathbb{R}^2\times \mathbb{R}_{+}$.
If one uses symmetric extension, by the above symmetric theorem, there exists not certainly symmetric solution.
That is to say, the energy inside $\Omega$ and outside $\Omega$ of the solution for the extended data are not certainly equal.
But if we assume the initial data satisfies certain symmetric properties, then the situation will be changed.
Let the real initial data $u_0(x)$ satisfies the following symmetry property.
\begin{equation}\label{om.sy}
\begin{array}{ccccc}
-u_{0,1}(-x_1, x_2, x_3) &= & u_{0,1}(x_1, -x_2, x_3)&=& u_{0,1}(x_1, x_2, x_3)\\
u_{0,2}(-x_1, x_2, x_3) &= & -u_{0,2}(x_1, -x_2, x_3)&=& u_{0,2}(x_1, x_2, x_3)\\
u_{0,3}(-x_1, x_2, x_3) &= & u_{0,3}(x_1, -x_2, x_3)&=& u_{0,3}(x_1, x_2, x_3),
\end{array}
\end{equation}
In this section, we search energy conservation smooth solution of \eqref{11.2} with domain $\Omega= \mathbb{R}^2\times \mathbb{R}_{+}$.
We will see such symmetry property is sufficient and necessary for smooth solution with energy conservation.
Denote
\begin{equation} \label{e.s.2}
P_s {u}_{0}(x)= \left\{
\begin{array}{cc}
u_0(x), & x_3\geq 0;\\
(u_{0,1}(\tilde{x}), u_{0,2}(\tilde{x}), -u_{0,3}(\tilde{x}))^{t}, & x_3<0.
\end{array}
\right.
\end{equation}
\begin{definition}
We say $u(x)\in (\dot{H}^{\frac{1}{2}}(\mathbb{R}^2\times \mathbb{R}_{+}))^{3}$, if
$P_s u(x)\in (\dot{H}^{\frac{1}{2}}(\mathbb{R}^{3}))^{3}$.
\end{definition}

As a byproduct of the study of symmetry property, I find out
smooth solution with energy conservation.
\begin{theorem} \label{th:main2}
Given $m>2$.
If the real initial data $u_0$ satisfies ${\rm div } \, u_0=0$, symmetry property \eqref{om.sy} and
$\|u_{0}\|_{(\dot{H}^{\frac{1}{2}}(\mathbb{R}^2\times \mathbb{R}_{+}))^{3}}$ being small,
then the Navier-Stokes equations \eqref{11.2} have a $C^{m-1}$ smooth solution $u(t,x)$ with energy conservation and satisfying symmetry property \eqref{om.sy}.
\end{theorem}
The above theorem \ref{th:main2} can be extend to Besov spaces, Triebel-Lizorkin spaces, Besov-Morrey spaces and Triebel-Lizorkin-Morrey spaces.
Further, our method can be applied also to the following domain
$\Omega= \mathbb{R}\times \mathbb{R}_{+}\times \mathbb{R}_{+}$
or $\mathbb{R}_{+}\times \mathbb{R}_{+}\times \mathbb{R}_{+}$.

In real world, when considering Navier-Stokes equations \eqref{11.2} on the domain $\Omega$, the initial data must be real valued and
the solution should keep energy conservation.
For general domain $\Omega$, to solve \eqref{11.2} with initial data $u_0$, we can try to use integration equation \eqref{1.3}.
For example, we can restrict the operators $e^{t\Delta}$, $e^{(t-s)\Delta}$ and $e^{(t-s)\Delta}(-\Delta)^{-1}$ in equations \eqref{1.4}
on the domain $\Omega\times \Omega$.
But, even $u^{\tau}(t,x)$ in \eqref{1.4} converges to some function $u(t,x)$, we did not know
whether $u(t,x)$ satisfies the first equation of equations \eqref{11.2}.

Another way to solve \eqref{11.2} is to extend the initial data $u_0$ to a function in $\mathbb{R}^{3}$.
Let $P_z{u}_{0}(x)$ be the zero extension of $u_0$.
\begin{equation*}
P_z{u}_{0}(x)= \left\{
\begin{array}{cc}
u_0(x), & x\in \Omega;\\
0, & x\notin \Omega.
\end{array}
\right.
\end{equation*}
We solve the equations \eqref{1.1}, we get a solution $P_{z}u(t,x)$. We say the restriction of $P_{z}u(t,x)$ on the domain $\Omega$ is the relative solution of \eqref{11.2}.
But we know $P_{z}u(t,x)$ can not be zero outside $\Omega$.
Hence there exists loss of energy. It is hard to say, $u(t,x)$ is the real solution of \eqref{11.2}.

Further, let $v(x)$ be a function with ${\rm supp} v(x) \subset \Omega^{c}$. We can consider other extension
$P_{z,v}{u}_{0}(x)$
which is the zero extension of $u_0$.
\begin{equation*}
P_{z,v}{u}_{0}(x)= \left\{
\begin{array}{cc}
u_0(x), & x\in \Omega;\\
v(x), & x\notin \Omega.
\end{array}
\right.
\end{equation*}
We can consider also the restriction of solution $P_{z,v}u(t, x)$ of \eqref{1.1} with initial data $P_{z,v}{u}_{0}(x)$.
There are many choice for $v(x)$. Generally speaking,
we have no sufficient reason to say that $P_{z,v}u(t, x)$ is the real solution of \eqref{11.2}.

Hence I consider particular domain $\Omega= \mathbb{R}^{2}\times \mathbb{R}_{+}$. I choose $v(x)$ to be the symmetric extension of $u_0(x)$.
That is to say, I extend $u_0(x)$ to $P_{z,v}u(t, x)$ in a symmetric way.
To ensure the zero divergence property ${\rm div} P_{z,v}u(t, x)=0$, I can not extend all the three $u_{0,\tau}(x)$ with the same symmetry method.
If I extend $u_{0,1}(x)$ and $u_{0,2}(x)$ with the anti-symmetric way and I have to extend $u_{0,3}(x)$ with symmetric way.
If I extend $u_{0,1}(x)$ and $u_{0,2}(x)$ with the symmetric way and I have to extend $u_{0,3}(x)$ with anti-symmetric way.
There are two ways to extend the initial value symmetrically.
Denote
$u_0(x) = (u_{0,1}(x), u_{0,2}(x), u_{0,3}(x))^{t}$ and $\tilde{x}=(x_1,x_2,-x_3)$.
The first way is to extend $u_0(x)$ in the following way:
\begin{equation} \label{e.s.1}
P_{as} {u}_{0}(x)= \left\{
\begin{array}{cc}
u_0(x), & x_3\geq 0;\\
(-u_{0,1}(\tilde{x}), -u_{0,2}(\tilde{x}), u_{0,3}(\tilde{x}))^{t}, & x_3<0.
\end{array}
\right.
\end{equation}
The second way is to extend $u_0(x)$ as I did in the equation \eqref{e.s.2}.
Then I can consider the solutions $P_{as} u(t,x)$ and $P_{s} u(t,x)$  of \eqref{1.1} with respectively initial data of $P_{as} {u}_{0}(x)$ and $P_{s} {u}_{0}(x)$.

For the first way, the iteration process \eqref{1.4} will change the symmetric property.
Given initial data $P_{as} {u}_{0}(x)$ defined in \eqref{e.s.1}.
If $P_{as} u(t,x)$ exists, according to Theorem \ref{th:main0},
$P_{as} u(t,x)$ can not be a symmetric solution of \eqref{1.1}.
We did not know whether the energy in the domain  $\mathbb{R}^{2}\times \mathbb{R}_{+}$ equals to
the energy in the domain  $\mathbb{R}^{2}\times \mathbb{R}_{-}$.

In this paper, I search energy conservation smooth solution of \eqref{11.2} with domain $\Omega= \mathbb{R}^2\times \mathbb{R}_{+}$.
To get smooth solution, the real initial data $u_0(x)$ should satisfy the symmetry property \eqref{om.sy}.
We will see such symmetry property is sufficient and necessary for smooth solution with energy conservation.
In fact, there exists and only exists one kind of symmetry property generate
smooth solution with energy conservation.
\begin{theorem} Given $m>2$.
If the real initial data $u_0$ satisfies ${\rm div } \, u_0=0$, symmetry property \eqref{om.sy} and
$\|u_{0}\|_{(\dot{H}^{\frac{1}{2}}(\mathbb{R}^2\times \mathbb{R}_{+}))^{3}}$ being small,
then the Navier-Stokes equations \eqref{11.2} have a $C^{m-1}$ smooth solution $u(t,x)$ with energy conservation and satisfying symmetry property \eqref{om.sy}.
\end{theorem}

\begin{proof}

If $u_0$ satisfies  the symmetry property \eqref{om.sy}, then
\begin{equation} \label{X9}
T(P_{s}u_{0})=
\left( \begin{aligned}
e_1  \\
e_2 \\
e_3
 \end{aligned}\right ).
\end{equation}
$u^{\tau}(t,x)$ in the iteration process \eqref{1.4} will keep the symmetric property
\begin{equation*}
T(u^{\tau}(t,x))=
\left( \begin{aligned}
e_1  \\
e_2 \\
e_3
 \end{aligned}\right ).
\end{equation*}

(i) $u_0\in (\dot{H}^{\frac{1}{2}}(\mathbb{R}^2\times \mathbb{R}_{+}))^{3}$ means $P_s{u}_{0}(x)\in (\dot{H}^{\frac{1}{2}}(\mathbb{R}^3))^{3}$.
For $\|u_0\|_{(\dot{H}^{\frac{1}{2}}(\Omega))^{3}}$ being small, $\|P_s{u}_0\|_{(\dot{H}^{\frac{1}{2}}(\mathbb{R}^3))^{3}}$ being small.
Hence $u^{\tau}(t,x)$ converges to some $u(t,x)\in (S_{m,m'}(\mathbb{R}^{3}))^{3}$ which is the solution of the equations \eqref{1.1}.
Since $S_{m,m'}(\mathbb{R}^{3})\subset C^{m-1}(\mathbb{R}^{3})$, by Theorems \ref{th:main0} and \ref{th:LXY},
we get smooth symmetric solution on the whole $\mathbb{R}^{3}$.
By restriction, we get also smooth solution on the domain $\mathbb{R}^{2}\times \mathbb{R}_{+}$.

(ii) Further, according to Theorem \ref{th:main0}, we know the energy in the domain  $\mathbb{R}^{2}\times \mathbb{R}_{+}$ equals to
the energy in the domain  $\mathbb{R}^{2}\times \mathbb{R}_{-}$.

By the above two points (i) and (ii), we get smooth solution with energy conservation.
\end{proof}

\begin{remark}
(i) For real valued initial data with symmetric solution, there exists only one possibility.
If $u_0$ does not satisfies \eqref{om.sy}, then  the iteration process \eqref{1.4}
will not keep the symmetry property and will not keep energy conservation.

(ii) General speaking, the fluid on the domain has ripple effect near the boundary.
We have proved that, if the initial data satisfies the symmetry property \eqref{om.sy}, then
the Navier-Stokes equations \eqref{11.2} have smooth solution even at the boundary.

(iii) Our method can be applied also to the following domain
$\Omega= \mathbb{R}\times \mathbb{R}_{+}\times \mathbb{R}_{+}$
or $\mathbb{R}_{+}\times \mathbb{R}_{+}\times \mathbb{R}_{+}$.
\end{remark}

\hspace{6cm}

{\bf Acknowledgement}\ \ The author is financially supported by the
National Natural Science Foundation of China  (No.11571261).

\end{document}